\documentclass[10pt]{amsart}
\usepackage{amsmath, amscd, amsthm} 
\usepackage{amssymb, amsfonts} 
\usepackage{enumerate}
\usepackage[svgnames]{xcolor} 
\usepackage{bbm}
\usepackage{color}
\usepackage[all]{xy} 
\usepackage[margin=1.5in]{geometry}
\usepackage{palatino}
\usepackage{mathrsfs}
\usepackage{booktabs}
\usepackage{aliascnt}
\usepackage{mathrsfs}
\usepackage{hyperref}
 \definecolor{dark-red}{rgb}{0.4,0.15,0.15}
\hypersetup{
    colorlinks, linkcolor=dark-red,
    citecolor=DarkBlue, urlcolor=MediumBlue
}
\usepackage{mathrsfs}
\usepackage[
textwidth=3cm,
textsize=small,
colorinlistoftodos]
{todonotes}

\usepackage{tikz}
\usepackage{tikz-cd}

\newcommand{\Q}{\mathbb{Q}} 
\newcommand{\QQ}{\Q}
\newcommand{\CC}{\mathbb{C}} 
\renewcommand{\AA}{\mathbb{A}}

\newcommand{\ZZ}{\mathbb{Z}}
\newcommand{\R}{\mathbb{R}}
\newcommand{\RR}{\R}

\newcommand{\FF}{\mathbb{F}}

\newcommand{\kk}{\mathsf{k}}

\newcommand{\sphere}{\mathrm{S}}
\newcommand{\slice}{\mathscr{S}}

\newcommand{\ul}[1]{\underline{\smash{#1}}}

\DeclareMathOperator{\Ext}{Ext}

\renewcommand{\top}{{\mathrm{top}}}
\newcommand{\Sq}{\mathsf{Sq}}
\newcommand{\jw}{\mathsf{jw}}
\newcommand{\KGL}{\mathsf{KGL}}

\usepackage[T1]{fontenc}

\numberwithin{equation}{section} 

\theoremstyle{plain}

\newaliascnt{theorem}{equation}  
\newtheorem{theorem}[theorem]{Theorem}  
\aliascntresetthe{theorem}

 \theoremstyle{definition}

\newaliascnt{prop}{equation}  
\newtheorem{prop}[prop]{Proposition}
\aliascntresetthe{prop}

\newaliascnt{lemma}{equation}  
\newtheorem{lemma}[lemma]{Lemma}
\aliascntresetthe{lemma}

\newaliascnt{corollary}{equation}  
\newtheorem{corollary}[corollary]{Corollary}
\aliascntresetthe{corollary}

\newaliascnt{claim}{equation}  

\aliascntresetthe{claim}

\newaliascnt{conjecture}{equation}  
\newtheorem{conjecture}[conjecture]{Conjecture}
\aliascntresetthe{conjecture}

\newaliascnt{question}{equation}  

\aliascntresetthe{question}

\newaliascnt{defn}{equation}  
\newtheorem{defn}[defn]{Definition}
\aliascntresetthe{defn}

\newaliascnt{example}{equation}  

\aliascntresetthe{example}

\theoremstyle{remark}

\newaliascnt{remark}{equation}  
\newtheorem{remark}[remark]{Remark}
\aliascntresetthe{remark}

\newaliascnt{convention}{equation}  

\aliascntresetthe{convention}

\theoremstyle{plain}
\newtheorem*{mainthm}{Main Theorem}

\newcommand{\aref}[1]{\autoref{#1}}

\renewcommand{\AA}{\mathbb{A}}
\newcommand{\SHA}{\mathrm{SH}^{\AA^1}\!}

\newcommand{\PP}{\mathbb{P}}
\newcommand{\GG}{\mathbb{G}}

\newcommand{\Spec}{\operatorname{Spec}}

\newcommand{\chr}{\operatorname{char}}
\newcommand{\cd}{\operatorname{cd}}
\renewcommand{\sc}{\operatorname{sc}}

\newcommand{\comp}[1]{^{\widehat{~}}_{#1}}

\newcommand{\kw}{\mathrm{kw}}
\newcommand{\KW}{\mathrm{KW}}
\newcommand{\KQ}{\mathrm{KQ}}
\newcommand{\Dd}{\mathrm{D}}

\begin{document}
\title{The homotopy groups of the $\eta$-periodic motivic sphere spectrum}

\author{Kyle Ormsby}
\address{Reed College}
\email{ormsbyk@reed.edu}

\author{Oliver R\"ondigs}
\address{Universit\"at Osnabr\"uck}
\email{oroendig@uni-osnabrueck.de}

\thanks{K.O.~was partially supported by NSF grant DMS-1709302; O.R.~was partially supported by the DFG SPP 1786.}

\begin{abstract}
We compute the homotopy groups of the $\eta$-periodic motivic sphere spectrum over a field $\kk$ of finite cohomological dimension with characteristic not $2$ and in which $-1$ a sum of four squares.  We also study the general characteristic $0$ case and show that the $\alpha_1$-periodic slice spectral sequence over $\QQ$ determines the $\alpha_1$-periodic slice spectral sequence over all extensions $\kk/\QQ$.  This leads to a speculation on the role of a ``connective Witt-theoretic $J$-spectrum'' in $\eta$-periodic motivic homotopy theory.
\end{abstract}

\maketitle

\section{Introduction}\label{sec:intro}

The motivic sphere spectrum $\sphere$ is the unit object in the tensor triangulated stable homotopy category of motivic spectra $(\SHA(\kk),\wedge)$ over a field $\kk$.  In this category, both the simplicial circle $S^1$ and the punctured affine line $\AA^1\smallsetminus 0$ are $\wedge$-invertible, so it is crucial that we understand the bigraded homotopy groups $\pi_\star\sphere := \bigoplus_{m,n\in \ZZ}\pi_{m+n\alpha}\sphere$ where $\pi_{m+n\alpha}\sphere := [(S^1)^{\wedge m}\wedge (\AA^1\smallsetminus 0)^{\wedge n},\sphere]_{\SHA(\kk)}$.  See the introduction to \cite{ORO:vanishing} for a more complete discussion of the importance of this ring.

The motivic Hopf map $\eta\in\pi_\alpha\sphere$ represented by the canonical $\GG_m$-torsor $\AA^2\smallsetminus 0\to \PP^1$ plays an especially important role in $\pi_\star\sphere$.  This class is non-nilpotent over all fields \cite{morel:pi0} and thus represents a first example of exotic behavior in $\pi_\star \sphere$, differentiating it from the classical stable stems.  (Recall that $\eta^4=0$ classically, and that the Nishida Nilpotence Theorem \cite{nishida} tells us that all classes of nonzero degree in the classical stable stems are nilpotent.)  Let
\[
  \eta^{-1}\sphere := \operatorname{hocolim}(\sphere\xrightarrow\eta \Sigma^{-\alpha}\sphere\xrightarrow\eta \Sigma^{-2\alpha}\sphere\xrightarrow\eta \Sigma^{-3\alpha}\sphere\xrightarrow\eta \cdots)
\]
denote the $\eta$-periodic sphere spectrum.\footnote{Other authors have referred to this object as the $\eta$-local or $\eta$-inverted sphere.  We have chosen our terminology to match the language of classical $v_n$-periodic homotopy theory, which seems appropriate given the emerging role of $\eta$ in motivic nilpotence and periodicity \cite{andrews:w,gheorghe:exotic}.}  We have $\pi_\star \eta^{-1}\sphere\cong \eta^{-1}\pi_\star\sphere$ (where the latter term represents the localization of the ring $\pi_\star\sphere$ at the multiplicative set $\{\eta,\eta^2,\ldots\}$), so inverting $\eta$ annihilates $\Gamma_\eta := \{x\in \pi_\star \sphere\mid x\eta^N=0\text{ for some }N\}$ and induces an injection $\pi_\star \sphere/\Gamma_\eta\to \pi_\star\eta^{-1}\sphere$.  

A number of authors have studied $\pi_\star\eta^{-1}\sphere$ over particular fields, including M.~Andrews and H.~Miller \cite{AM} over $\CC$, B.~Guillou and D.~Isaksen \cite{GI:etaC,GI:etaR} over $\CC$ and $\RR$, and G.~Wilson \cite{Wilson} over finite fields, local fields, and $\QQ$.  Over $\CC$, $\pi_{m+n\alpha}\sphere\cong W(\CC)\cong \ZZ/2$ for nonnegative $m$ congruent to $0$ or $3$ mod $4$, whereas more complicated ``image of $J$''-style patterns occur in $\pi_\star \eta^{-1}\sphere\comp{2}$ (the bigraded homotopy groups of the $\eta$-periodic $2$-complete sphere) over $\RR$ and $\QQ$.  These authors work with either the motivic Adams-Novikov or motivic Adams spectral sequence in order to produce their results.  In addition to these results, R\"ondigs \cite{roendigs:etainv} has shown that $\pi_1\eta^{-1}\sphere=\pi_2\eta^{-1}\sphere=0$ over all fields of characteristic different from $2$.

In this note, we use the $\alpha_1$-periodic slice spectral sequence to completely determine $\pi_\star \eta^{-1}\sphere$ over finite-cohomological dimension fields with characteristic different from $2$ in which $-1$ as a sum of four squares.\footnote{The smallest $r$ such that $-1$ is a sum of $r$ squares in $\kk$ is called the \emph{level} of $\kk$ and is often denoted $s(\kk)$. (The $s$ is for \emph{Stufe}).  By a theorem of Pfister \cite{pfister:level}, $s(\kk)$ is always a power of $2$.  See \cite[Examples XI.2.4]{lam:intro} for examples of fields of various levels.}  Let $W(\kk)$ denote the Witt ring of quadratic forms over $\kk$ modulo the hyperbolic plane.

\begin{mainthm}[see \aref{thm:main}]\label{thm:start}
Let $\kk$ be a field of finite cohomological dimension with characteristic not $2$.  If $-1$ is a sum of four squares in $\kk$, then
\[
  \pi_\star \eta^{-1}\sphere \cong W(\kk)[\eta^{\pm 1},\sigma,\mu]/(\sigma^2)
\]
where $|\sigma|=3+4\alpha$ and $|\mu| = 4+5\alpha$.  In particular, the bigraded homotopy groups of $\eta^{-1}\sphere$ are
\[
  \pi_{m+n\alpha}\eta^{-1}\sphere\cong
  \begin{cases}
    W(\kk)&\text{if }m\ge 0\text{ and }m\equiv 0\text{ or }3\pmod{4}\text{,}\\
    0&\text{otherwise.}
  \end{cases}
\]
\end{mainthm}

In \aref{cor:nonperiodic} (see also \cite[Theorem 5.5]{ORO:vanishing}), we see that for fields satisfying the same hypotheses, $\pi_{m+n\alpha}\sphere\cong \pi_{m+n\alpha}\eta^{-1}\sphere$ for $2n\ge \max\{3m+5,4m\}$, so we have also computed a bi-infinite range of homotopy groups of the motivic sphere spectrum.

The picture is less clear for fields of characteristic $0$ with cohomological dimension greater than $2$, but we are able to produce some partial results in \aref{sec:char0}.  Let $\alpha_1^{-1}\slice(\kk)$ denote the $\alpha_1$-periodic slice spectral sequence over $\kk$.  In \aref{thm:profile}, we show that $\alpha_1^{-1}\slice(\QQ)$ determines $\alpha_1^{-1}\slice(\kk)$ for any field extension $\kk/\QQ$.  This leads to a conjecture on the differentials in $\alpha_1^{-1}\slice(\kk)$ and some speculations regarding the structure of $\pi_\star \eta^{-1}\sphere$ in general.

\subsection*{Acknowledgments}
The first named author thanks Paul Arne {\O}stv{\ae}r, Glen Wilson, and participants at the conference \emph{Motivic Homotopy Groups of Spheres III} who gave particularly helpful feedback during an early stage of this project.  Both authors thank Tom Bachmann for a helpful discussion regarding \aref{conj:jw}.

The spectral sequence diagrams in this paper were produced with Hood Chatham's \LaTeX\ package \href{https://www.ctan.org/pkg/spectralsequences}{\texttt{spectralsequences}}.

\section{The $\alpha_1$-periodic slice spectral sequence}\label{sec:slice}

In this section, we set up the $\alpha_1$-periodic slice spectral sequence and discuss its convergence properties and first two pages over a general field of characteristic different from $2$.

We refer to \cite{RSO:pi1} for the setup of the slice spectral sequence, and only briefly recall some pertinent facts here.  The sphere spectrum $\sphere$ is effective and thus has a slice tower of the form
\[
\begin{tikzcd}
\cdots\rar &f_3\sphere\dar\rar &f_2\sphere\dar\rar &f_1\sphere\dar\rar &f_0\sphere=\sphere\dar\\
&s_3\sphere &s_2\sphere &s_1\sphere &s_0\sphere
\end{tikzcd}
\]
where $f_q\sphere$ is the $q$-th effective cover of $\sphere$ and $s_q\sphere$ is the $q$-th slice of $\sphere$.  Associated with this tower is the slice spectral sequence $\slice$ with $E_1$-page $\slice_1^{q,m+n\alpha} = \pi_{m+n\alpha}s_q\sphere$.  (If we need to refer to the base field, then we will denote this spectral sequence $\slice(\kk)$.)  The differentials take the form $d_r:\slice_r^{q,m+n\alpha}\to \slice_r^{q+r,m-1+n\alpha}$.

The following theorem states the basic convergence properties of $\slice$; it is a concatenation of \cite[Theorem 3.50]{RSO:pi1} and \cite[Theorem 1.3]{ORO:vanishing}. Recall that the $\eta$-complete sphere spectrum is $\sphere\comp{\eta} := \operatorname{holim}\sphere/\eta^n$ where $\sphere/\eta^n$ is the cofiber of $\eta^n:\Sigma^{n\alpha}\sphere\to \sphere$.

\begin{theorem}\label{thm:sliceCond}
The slice spectral sequence for $\sphere$ conditionally converges to $\pi_\star \sphere\comp{\eta}$.  Moreover, if $\cd k<\infty$, then $\sphere\comp{\eta}\simeq \sphere$ and the slice spectral sequence for $\sphere$ converges conditionally to $\pi_\star \sphere$.
\end{theorem}

We also have control over $\slice_1$ via the following slice computation.  We let $\Ext_{MU_*MU}^{s,t}$ denote the cohomology of the $MU$-Hopf algebroid in cohomological degree $s$ and internal degree $t$.  (Recall that this is the $E_2$-page of the Novikov, \emph{i.e.}, $MU$-Adams, spectral sequence from topology.)

\begin{theorem}[{\cite[Theorem 2.2]{RSO:pi1}}]\label{thm:slices}
The $q$-th slice of the motivic sphere spectrum is
\[
  s_q\sphere \simeq \bigvee_{s\ge 0} \Sigma^{q-s+q\alpha}H\Ext_{MU_*MU}^{s,2q}
\]
at least after inverting the exponential characteristic of the
base field.
\end{theorem}

We refer to \cite{ravenel} for basic facts about $\Ext_{MU_*MU}$, and we use its naming conventions for elements.  Importantly, there is a single nonzero class $\alpha_1\in \Ext_{MU_*MU}^{1,2}$ that represents $\eta$.  Multiplication by $\alpha_1$ induces a map of spectral sequences $\alpha_1:\slice^{q,m+n\alpha}\to\slice^{q+1,m+(n+1)\alpha}$.  Taking the colimit of the tower given by iterating this map produces the $\alpha_1$-periodic slice spectral sequence, $\alpha_1^{-1}\slice$.  We will analyze the target and convergence properties of $\alpha_1^{-1}\slice$ momentarily, but it certainly appears that this construction ought to say something about $\eta^{-1}\sphere$ or a related object.

There is another obvious spectral sequence we could consider, namely the slice spectral sequence for $\eta^{-1}\sphere$, but it turns out that the two spectral sequences are the same.  For a motivic spectrum $X$, recall from \cite[Definition 3.1]{RSO:pi1} that $\sc(X)$ is the \emph{slice completion} of $X$.

\begin{theorem}\label{thm:etaSlices}
The slice spectral sequence for $\eta^{-1}\sphere$ and $\alpha_1^{-1}\slice$ are isomorphic as spectral sequences.  They both have first page additively isomorphic to $\pi_\star H\FF_2[\alpha_1^{\pm 1},\alpha_3,\alpha_4]/(\alpha_4^2)$ and conditionally converge to $\pi_\star\sc(\eta^{-1}\sphere)$ in the sense of \cite{boardman}.  If $\cd\kk<\infty$, then $\sc(\eta^{-1}\sphere)\simeq \eta^{-1}\sphere$ and the isomorphic spectral sequences conditionally converge to $\pi_\star\eta^{-1}\sphere$.
\end{theorem}

\begin{remark}
In \aref{cor:fdconv} and \aref{sec:comp}, we will see that convergence is in fact strong if $\kk$ has odd characteristic or $\cd_2\kk<\infty$.  In \aref{sec:char0}, we will see that convergence over characteristic $0$ fields to $\pi_\star \sc(\eta^{-1}\sphere)$ is strong.
\end{remark}

\begin{remark}\label{rmk:mult}
As a ring object, $s_*\eta^{-1}\sphere$ is not an $H\FF_2$-algebra \cite[Remark 2.33]{RSO:pi1}, and our identification of $\alpha_1^{-1}\slice_1$ in \aref{thm:etaSlices} is not multiplicative.  By a bi-degree argument and the general properties of slice multiplicativity given in \cite[Section 2.4]{RSO:pi1}, the multiplication on $\alpha_1^{-1}\slice_1$ agrees with the ``naive'' multiplication up to addition of some terms involving $\Sq^1$.  Our determination of $\alpha_1^{-1}\slice_2$ does not depend on the precise multiplicative structure, and we will see in \aref{thm:d1} that the multiplication on $\alpha_1^{-1}\slice_2$ is fairly simple.
\end{remark}

\begin{proof}[Proof of \aref{thm:etaSlices}]
Let $E$ denote the slice spectral sequence for $\eta^{-1}\sphere$.  Then
\[
  E_1\cong \pi_\star H\FF_2[\alpha_1^{\pm 1},\alpha_3,\alpha_4]/(\alpha_4^2)
\]
by \cite[Theorem 2.35]{RSO:pi1}.\footnote{Although \aref{thm:slices} holds
after inverting the exponential characteristic, the slices $s_*\eta^{-1}\sphere$
are known without inverting the exponential characteristic. The reason is that
if $k$ is a field of odd characteristic $p$, then multiplication with $p$
is an isomorphism on the Witt ring of $k$, and hence on $\eta^{-1}\sphere$.} In particular, the canonical map $\slice\to E$ takes $\alpha_1$ to a unit and hence induces a map $\alpha_1^{-1}\slice\to E$.  By \cite[Corollary 6.2.3]{AM}, $\alpha_1^{-1}\Ext_{MU_*MU}\cong \FF_2[\alpha_1^{\pm 1},\alpha_3,\alpha_4]/(\alpha_4^2)$.  Given this result and the form of $\slice_1$ in \aref{thm:slices}, we conclude that $\alpha_1^{-1}\slice_1\to E_1$ is an isomorphism, and it follows that $\alpha_1^{-1}\slice\cong E$.

Our first convergence statement is formal given the construction of the slice spectral sequence (see \cite[\S 3.1]{RSO:pi1}).  For the second convergence statement, \aref{thm:slices} tells us that $\sc(\sphere)\simeq\sphere$ when $\cd \kk<\infty$.  Given the conditional convergence conditions of \cite[Definition 5.10]{boardman}, our result follows as long as the sequential colimit that inverts $\alpha_1$ commutes with the limit defining slice completion.  Our assumption on cohomological dimension implies a vanishing line parallel to $\alpha_1$-multiplication, and hence the limit in question is finite and commutes with sequential colimits.
\end{proof}

This leads us to the main theorem of this section, a determination of the first slice differentials and $\alpha_1^{-1}\slice_2$:

\begin{theorem}\label{thm:d1}
Over any field $\kk$ of characteristic different from $2$, the first slice differential for $\eta^{-1}\sphere$ has the form
\[
\begin{pmatrix}
  \Sq^2      &0 &\tau           &0     &0    &0 &0 &\cdots \\
  0          &0 &0              &0     &0    &0 &0 &\cdots\\
  \Sq^3\Sq^1 &0 &\Sq^2+\rho\Sq^1 &0     &0    &0 &0 &\cdots\\
  0          &0 &0              &\Sq^2     &0    &\tau &0 &\cdots\\
  0          &0 &\Sq^3\Sq^1     &\Sq^2\Sq^1+\Sq^3 &\Sq^2 &\rho &\tau &\cdots\\
  0          &0 &0             &\Sq^3\Sq^1 &0 &\Sq^2+\rho\Sq^1 &0 &\cdots\\
  0          &0 &0     &0  &\Sq^3\Sq^1&\Sq^2\Sq^1+\Sq^3&\Sq^2+\rho\Sq^1&\cdots\\
  0&0&0&0&0&\Sq^3\Sq^1&0&\cdots\\
  0&0&0&0&0&0&\Sq^3\Sq^1&\cdots\\
  \vdots &\vdots &\vdots &\vdots &\vdots &\vdots &\vdots &\ddots
\end{pmatrix}.
\]
Here the $i$-th column, $i\ge 0$, gives the first slice differential restricted to the summand $\Sigma^{i+q\alpha}H\FF_2$ of $s_q\eta^{-1}\sphere$ (the summand is $0$ if $i=1$).  The $j$-th row, $j\ge 0$, describes the incoming first slice differential for the summand $\Sigma^{j+1+(q+1)\alpha}H\FF_2$ of $\Sigma s_{q+1}\eta^{-1}\sphere$ (the summand is $0$ if $j=1$).

This results in an isomorphism of $k^M_*$-algebras
\[
  \alpha_1^{-1}\slice_2 \cong k^M_*[\alpha_1^{\pm 1},\alpha_4,\alpha_5]/(\alpha_4^2)
\]
where $|\alpha_4| = (4,3+4\alpha)$ and $\alpha_5 = \alpha_3^2\alpha_1^{-1}$ has degree $(5,4+5\alpha)$.
\end{theorem}

\begin{remark}
The factor of $\alpha_1^{-1}$ in the definition of $\alpha_5$ is not strictly necessary, but is there so that $\bar{\alpha}_5\mapsto \alpha_5$ under the localization map $\slice\to\alpha_1^{-1}\slice$.
\end{remark}

\begin{remark}
The determination
of the first slice differential for $\eta^{-1}\sphere$ complements the
occurrences of multiplications with $\tau$, which were used on p.11
of [17] to deduce vanishing columns in the in the Andrews-Miller range 
of the unlocalized slice spectral sequence.
\end{remark}

The pattern of differentials indicated by \aref{thm:d1} is represented graphically in \aref{fig:E1}.  The form of $\alpha_1^{-1}\slice_2$ also implies an important convergence result, which we state presently.

\begin{figure}
\begin{center}
\includegraphics[width=5in]{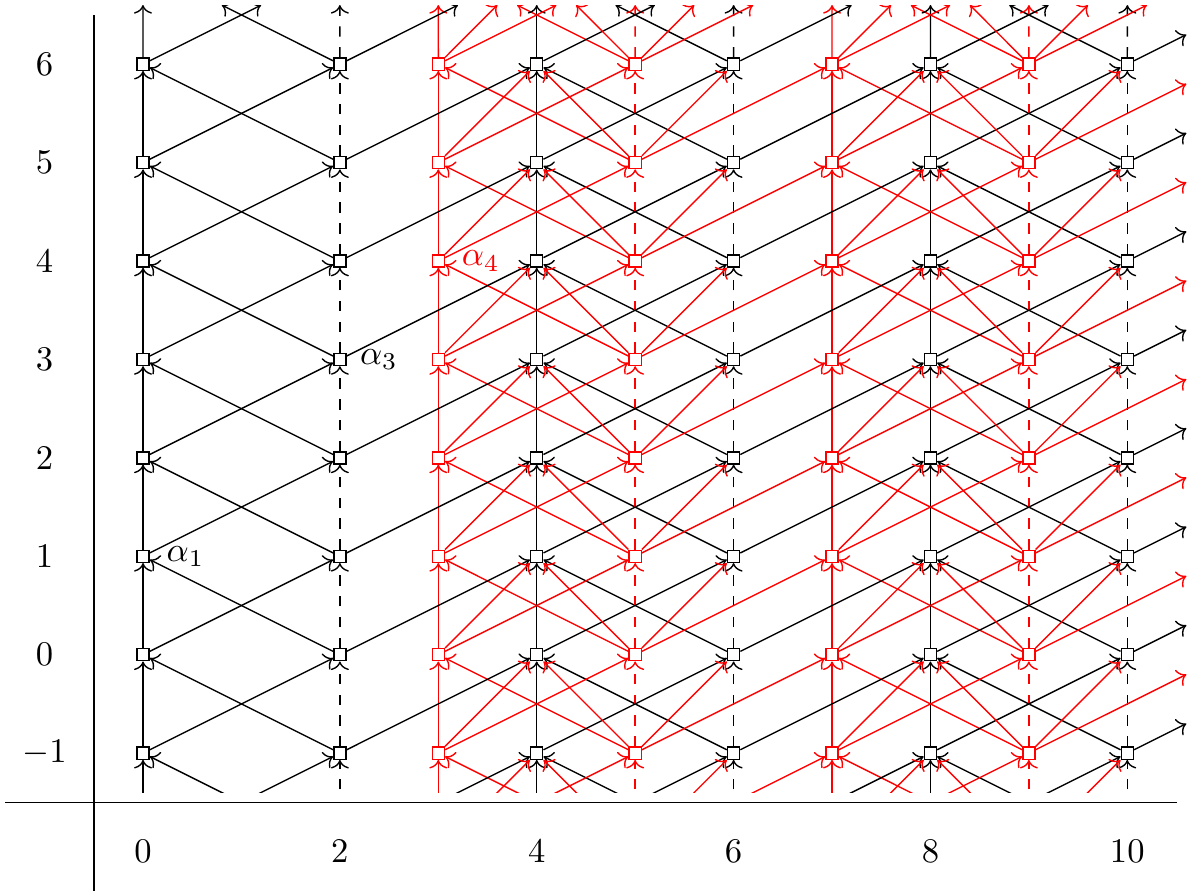}
\end{center}
\caption{The first page of the $\alpha_1$-periodic slice spectral sequence with its differentials.  A $\Box$ in position $(m,n)$ represents a copy of $\pi_\star \Sigma^{m+n\alpha}H\FF_2\cong \Sigma^{m+n\alpha}k^M_*[\tau]$ in slice degree $n$. The black portions of the diagram are detected by the unit map $\eta^{-1}\sphere\to \kw$, and the red portions are induced by $\sigma_\infty:\Sigma^3\kw\to \eta^{-1}\sphere$.  Arrows with slope $-1/2$ represent $\tau$, arrows with slope $-1$ represent $\rho$, solid vertical arrows represent $\Sq^2$, dashed vertical arrows represent $\Sq^2+\rho\Sq^1$, arrows with slope $1$ represent $\Sq^2\Sq^1+\Sq^3$, and arrows with slope $1/2$ represent $\Sq^3\Sq^1$.}\label{fig:E1}
\end{figure}

\begin{corollary}\label{cor:fdconv}
If $\cd_2(\kk)=r<\infty$, then $\alpha_1^{-1}\slice(\kk)$ collapses at its $(r+1)$-th page and converges strongly to $\pi_\star(\eta^{-1}\sphere)$.
\end{corollary}
\begin{proof}
The form of $\alpha_1^{-1}\slice_2$ (which is presented graphically in \aref{fig:E2}) and the fact that $k^M_{>r}(\kk)=0$ imply that $d_{>r}=0$.  This collapse along with the conditional convergence of \aref{thm:sliceCond} imply the strong convergence portion of the corollary.
\end{proof}

Our proof of \aref{thm:d1} requires a fascinating detour through connective Witt $K$-theory.  Let $\KQ$ denote the motivic spectrum representing Hermitian $K$-theory,\footnote{The $\mathrm{Q}$ is for $\mathrm{Q}$uadratic.} let $\KW = \eta^{-1}\KQ$ denote the $\eta$-periodization of $\KQ$, which is the motivic spectrum representing Balmer's higher Witt groups, and let $\kw$ denote the connective cover of $\KW$ (in the sense of Morel's homotopy $t$-structure).

The slices of $\kw$, the $d_1$-differentials of its slice spectral sequence, and the effect of $s_*$ on the unit map $\eta^{-1}\sphere\to \kw$ are known.

\begin{theorem}[{\cite[Theorem 19]{ARO:veryEffective}}]\label{thm:kw_slice}
Suppose the base scheme $S$ is a scheme over $\Spec(\ZZ[1/2])$.  Then
\[
  s_*\kw \simeq H\ZZ[\eta^{\pm 1},\sqrt{\beta}]/(2\eta,2\sqrt{\beta},\eta^2\xrightarrow{\Sq^1}\sqrt{\beta})
\]
where $|\eta| = \alpha$, $|\sqrt{\beta}| = 2+2\alpha$, and the first slice differential takes the form
\[
\begin{pmatrix}
  \Sq^2 &0 &\tau &0 &0 &0 &0 &\cdots \\
  0 &0 &0 &0 &0 &0 &0 &\cdots\\
  \Sq^3\Sq^1 &0 &\Sq^2+\rho\Sq^1 &0 &0 &0 &0 &\cdots\\
  0 &0 &0 &0 &0 &0 &0 &\cdots\\
  0 &0 &\Sq^3\Sq^1 &0 &\Sq^2 &0 &\tau &\cdots\\
  0 &0 &0 &0 &0 &0 &0 &\cdots\\
  0&0&0&0&\Sq^3\Sq^1&0&\Sq^2+\rho\Sq^1&\cdots\\
  0&0&0&0&0&0&0&\cdots\\
  0&0&0&0&0&0&\Sq^3\Sq^1&\cdots\\
  \vdots &\vdots &\vdots &\vdots &\vdots &\vdots &\vdots &\ddots
\end{pmatrix}
\]
(with the same conventions as \aref{thm:d1}).  Moreover, there is a splitting of $s_0\eta^{-1}\sphere$ such that the unit $\eta^{-1}\sphere\to \kw$ induces an inclusion on every even summand, and $\Sq^1$ on every odd summand.
\end{theorem}

\begin{proof}
  The description of the slices, as well as their multiplicative
  structure, is given in \cite[Theorem 19]{ARO:veryEffective}. The
  behaviour of the unit map follows from \cite[Lemmas 2.28, 2.29]{RSO:pi1}.
\end{proof}

Note that $s_*\kw$ and the pattern of $d_1$ differentials is precisely the black portion of \aref{fig:E1}.  The remaining portion of $s_*\eta^{-1}S$ (the red part of \aref{fig:E1}) is handled by the following theorem.

\begin{theorem}\label{thm:sigma_infty}
Over $\kk=\CC$, there is a unique homotopy class $\sigma_\infty:\Sigma^3 \kw\to \eta^{-1}\sphere$ inducing an isomorphism on $\Pi_3$.  This map induces 
$(1,\Sq^1)$ on every summand of a slice.
\end{theorem}

\begin{proof}
Fix $\kk=\CC$.  Recall from \cite[Section 4]{roendigs:etainv} that there is a cell presentation of $\kw$ over $\CC$ of the following form.  Namely, there is a sequence of cellular motivic spectra factoring the unit of $\kw$ as
\[
  \eta^{-1}\sphere = \Dd_1\xrightarrow{i_1} \Dd_2\xrightarrow{i_2} \cdots \to \Dd_n\xrightarrow{i_n} \cdots \to \kw
\]
such that for every $n$ the map $\Dd_n\to \kw$ is $(4n-1)$-connective and the composition $\eta^{-1}\sphere\to \Dd_n\to \kw$ induces isomorphisms on $\Pi_{4k}$.  For every $n\ge 1$, there is a unique nontrivial class $a_n:\Sigma^{4n-1}\eta^{-1}\sphere\to \Dd_n$ in $\pi_{4n-1}\Dd_n\cong \pi_{4n-1}\eta^{-1}\sphere$ such that
\[
  \Sigma^{4n-1}\eta^{-1}\sphere\xrightarrow{a_n} \Dd_n\xrightarrow{i_n} \Dd_{n+1}\xrightarrow{c_n} \Sigma^{4n}\eta^{-1}\sphere
\]
is a homotopy cofiber sequence with $c_n$ inducing an isomorphism on $\Pi_{4k+3}$ whenever $k\ge n$.  Taking the colimit as $n\to \infty$ gives a cell presentation of $\kw$.

We now construct $\sigma_\infty$.  Consider the map $\sigma_1=\eta^{-4}\sigma:\Sigma^3\eta^{-1}\sphere = \Sigma^3\Dd_1\to \eta^{-1}\sphere$.  Assume for induction that for some $n\ge 1$ a map $\sigma_n:\Sigma^3\Dd_n\to \eta^{-1}\sphere$ is given such that
\begin{enumerate}[(1)]
\item $\sigma_n i_{n-1} = \sigma_{n-1}$ and
\item $[\Sigma^4a_{n-1},\eta^{-1}\sphere]:[\Sigma^4\Dd_{n-1},\eta^{-1}\sphere]\to [\Sigma^{4n-1}\eta^{-1}\sphere,\eta^{-1}\sphere]$ is an isomorphism.
\end{enumerate}
Then the cofiber sequence above induces a long exact sequence
{\scriptsize
\[
  [\Sigma^{4n+2}\eta^{-1}\sphere,\eta^{-1}\sphere]\xleftarrow{a_n} [\Sigma^3 \Dd_n, \eta^{-1}\sphere]\xleftarrow{i_n} [\Sigma_3\Dd_{n+1},\eta^{-1}\sphere]\xleftarrow{c_n} [\Sigma^{4n+3}\eta^{-1}\sphere,\eta^{-1}\sphere]\xleftarrow{a_n} [\Sigma^4 \Dd_n, \eta^{-1}\sphere]\leftarrow \cdots.
\]
}
The Andrews-Miller theorem on $\pi_\star\eta^{-1}\sphere$ \cite{AM} implies
\[
  0 = [\Sigma^{4n+2}\sphere,\eta^{-1}\sphere] = [\Sigma^{4n+2}\eta^{-1}\sphere,\eta^{-1}\sphere],
\]
showing that $\sigma_n$ lifts to a map $\sigma_{n+1}$ such that $\sigma_{n+1}i_n=\sigma_n$.

Now note that assumption (2) implies that
\[
  [\Sigma^4 c_{n-1},\eta^{-1}\sphere]:[\Sigma^{4n}\eta^{-1}\sphere,\eta^{-1}\sphere]\to [\Sigma^4\Dd_n,\eta^{-1}\sphere]
\]
is surjective; furthermore, the composition
\[
  [\Sigma^{4n}\eta^{-1}\sphere,\eta^{-1}\sphere]\xrightarrow{[\Sigma^4c_{n-1},\eta^{-1}\sphere]} [\Sigma^4\Dd_n,\eta^{-1}\sphere]\xrightarrow{[\Sigma^4a_n,\eta^{-1}\sphere]} [\Sigma^{4n+3}\eta^{-1}\sphere,\eta^{-1}\sphere]
\]
is the map sending $\eta^{-5n}\mu_9^n$ to $\eta^{-5n-4}\sigma\mu_9^n$, hence an isomorphism.  It follows that the map $[\Sigma^4a_n,\eta^{-1}\sphere]$ is an isomorphism, as desired.  (In fact, we also get that there is a \emph{unique} $\sigma_{n+1}:\Sigma^3\Dd_{n+1}\to \eta^{-1}\sphere$ such that $\sigma_{n+1}i_n=\sigma_n$.)

Induction and the universal property of colimits now produces a map $\sigma_\infty:\Sigma^3\Dd_\infty\simeq \Sigma^3\kw\to \eta^{-1}\sphere$ sending $1\in \pi_3\Sigma^3\kw$ to $\eta^{-4}\sigma\in \pi_3\eta^{-1}\sphere$.  The uniqueness of $\sigma_\infty$ follows from the Milnor exact sequence and the vanishing of $\lim_n^1 [\Sigma^4\Dd_n,\eta^{-1}\sphere]$ (every group $[\Sigma^4\Dd_n,\eta^{-1}\sphere]$ beging finite of order $2$).

Since $\sigma_\infty$ is a map of $\eta^{-1}\sphere$-modules, it induces isomorphisms on $\Pi_{4m+3}$ for every integer $m$.  The statement on slices follows from
the behaviour of the unit map $\eta^{-1}\sphere \to \kw$ on slices given
in \aref{thm:kw_slice}.
\end{proof}

\begin{proof}[Proof of \aref{thm:d1}]
As $\eta^{-1}\sphere$ and $s_*\eta^{-1}\sphere$ are invariant under base change, it suffices to determine the first slice differential $d_1$ over $\ZZ[1/2]$. 
On a summand $\Sigma^nH\FF_2$, it is of the form
\[ (x_{n}\tau,a_{n}\tau\Sq^1+\beta_{n},b_{n}\Sq^2+\gamma_{n},c_{n}\Sq^2\Sq^1+d_{n}\Sq^3,e_{n}\Sq^3\Sq^1)\]
with $x_{n},a_{n},b_{n},c_{n},d_{n},e_{n}$ elements in $\ZZ/2$, and 
$\beta_{n},\gamma_{n}$ square classes
of units in $\ZZ[1/2]$. 
The behavior of the unit map $\eta^{-1}\sphere\to \kw$ on slices from 
\aref{thm:kw_slice} provides immediate restrictions:
\begin{align*}
  x_{4n}&=0, &a_{4n}\rho+\beta_{4n}&=\gamma_{4n}, &b_{4n}&=1, &e_{4n}&=1+c_{4n},\\ 
x_{4n+1}+a_{4n+1}&=1,  &(a_{4n+1}+1)\rho&=\beta_{4n+1},  &c_{4n+1}&=1,  &b_{4n+1}&=d_{4n+1},\\ 
 x_{4n+2}&=1,  &(a_{4n+2}+1)\rho+\beta_{4n+2}&=\gamma_{4n+2},  &b_{4n+2}&=1,  &e_{4n+2}&=1+c_{4n+2},\\ 
 x_{4n+3}&=a_{4n+3},  &a_{4n+3}\rho&=\beta_{4n+3},  &c_{4n+3}&=1,  &b_{4n+3}&=d_{4n+3}.
\end{align*}
The map $\sigma_\infty:\Sigma^3\kw_\CC\to \eta^{-1}\sphere_\CC$ on 
slices from \aref{thm:sigma_infty} imposes further restrictions:
\begin{align*}
x_{4n+1}&=1, &a_{4n+1}&=0, &b_{4n+1}&=c_{4n+1}=d_{4n+1}=1, \\e_{4n+1}&=1+d_{4n+2},
&x_{4n+3}&=0=a_{4n+3}, &b_{4n+3}&=c_{4n+3}=d_{4n+3}=1, &e_{4n+3}&=1+d_{4n+4}.
\end{align*}
Base change and the previous equations then provide the following equations:
\[\beta_{4n+1}=\rho, \quad\gamma_{4n+1}+\beta_{4n+2}=\rho, \quad\beta_{4n+3}= 
\gamma_{4n+3}=\gamma_{4n} =\beta_{4n}=\beta_{4n+2}=0.\]
Since the composition $d_1\circ d_1=0$, Adem relations
imply further coefficients. Considering the component
\[ d_1^2\colon \Sigma^{4n+2}H\FF_2 \to \Sigma^{4n+1+2\alpha }H\FF_2\]
implies that $a_{4n+2}=a_{4n}=0$. 
Considering the component
\[ d_1^2\colon \Sigma^{4n}H\FF_2 \to \Sigma^{4n+2+2\alpha }H\FF_2\]
provides that $c_{4n}=d_{4n}=0$, and hence $e_{4n}=1=e_{4n+3}$. 
The similar component 
\[ d_1^2\colon \Sigma^{4n+2}H\FF_2 \to \Sigma^{4n+4+2\alpha }H\FF_2\]
gives only $c_{4n+2}=d_{4n+2}$, and hence $e_{4n+1}=e_{4n+2}$.
Resorting to $s_\ast \sphere$ provides the solution $c_{4n+2}=d_{4n+2}=0$
and $e_{4n+1}=e_{4n+2}=1$ as follows. Consider the summand
$\Sigma^{2n+(2n+1)\alpha}H\FF_2$ in $s_{2n+1}\sphere$ generated by $\alpha_{2n+1}$. 
The first slice differential maps it via $\mathrm{inc}^2_{y_{2n+2}}\Sq^2\Sq^1$
to the top degree summand $\Sigma^{2n+1+(2n+2)\alpha} H\ZZ/(y_{2n+2})$ in
$s_{2n+2}\sphere$ by \cite[Lemma 4.1]{RSO:pi1}; here $y_{2n+2}$ is the
order of a cyclic group and divisible by four.
The map 
\[ \Sigma^\alpha \Sigma^{2n+1+(2n+2)\alpha} H\ZZ/(y_{2n+2}) 
\to \Sigma^{2n+1+(2n+3)\alpha} H\FF_2 \]
induced by $\eta$ is the projection $\mathrm{pr}^{y_{2n+2}}_2$. Hence
already after one multiplication with $\eta$, the 
degree $2+\alpha$ part of the first differential is zero on that
summand. It follows 
that $c_{2n}=0$. 

Given the form of the differentials, the additive calculation of $\alpha_1^{-1}\slice_2$ is nearly the same as the proof of \cite[Theorem 6.3]{RO:hermitian}.  The exotic multiplication on $\alpha_1^{-1}\slice_1$ mentioned in \aref{rmk:mult} reduces to $k^M_*[\alpha_1^{\pm 1},\alpha_4,\alpha_5]/(\alpha_4^2)$ in the subquotient $\alpha_1^{-1}\slice_2$ since $\Sq^1$ is only potentially nonzero on terms involving an odd power of $\tau$, and there are no $\tau$'s in $\alpha_1^{-1}\slice_2$.
\end{proof}

\section{Computations for fields with odd characteristic or cohomological dimension at most 1}\label{sec:comp}

Given the form of $\alpha_1^{-1}\slice_2$ and the spectral sequence's convergence properties determined in the previous section, we can now make short work of the following computations.

\begin{prop}\label{thm:collapse}
If $\cd_2\kk\le 1$, then the $\alpha_1$-periodic slice spectral sequence for $\sphere$ collapses with $\alpha_1^{-1}\slice_2=\alpha_1^{-1}\slice_\infty$ and converges strongly to $\pi_\star \eta^{-1}\sphere$.
\end{prop}
\begin{proof}
This is a specialization of \aref{cor:fdconv}.
\end{proof}

\begin{figure}
\begin{center}
\includegraphics[width=5in]{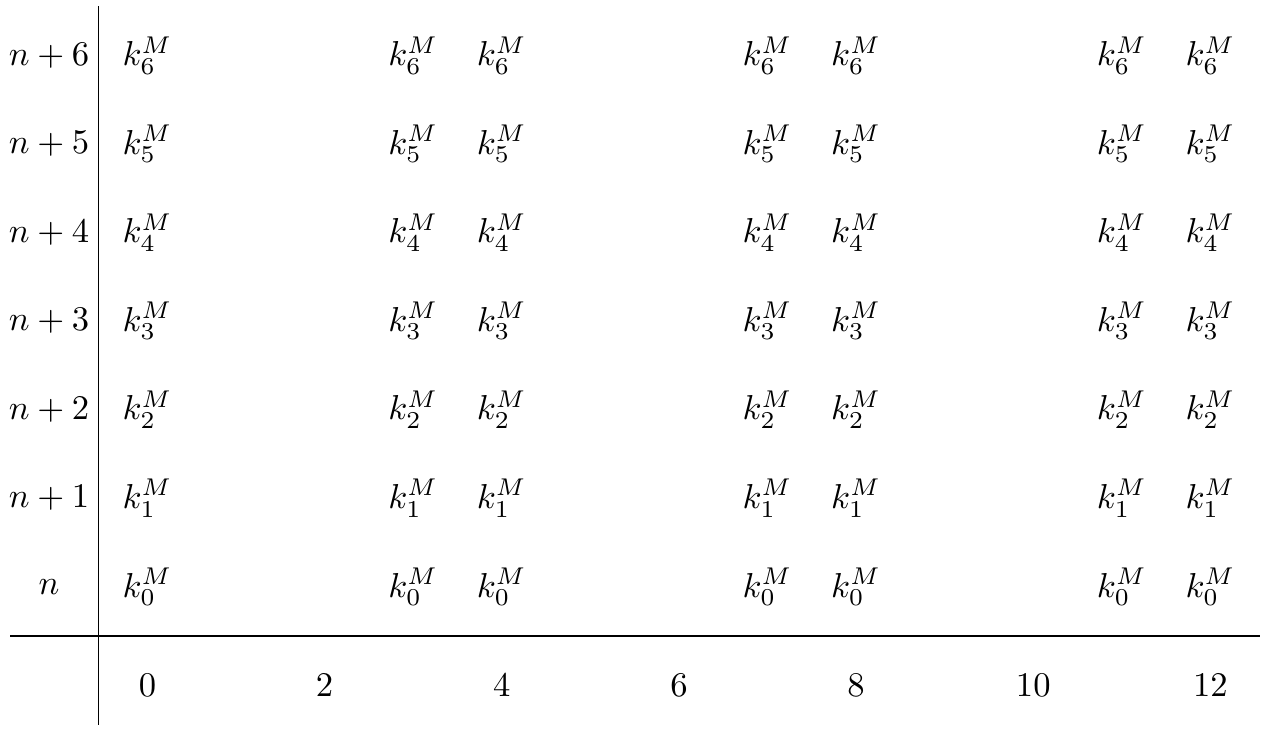}
\end{center}
\caption{The weight $n$ portion of $\alpha_1^{-1}\slice_2$ with simplicial degree on the horizontal axis and slice degree on the vertical axis.  This page is $0$ below slice degree $n$ and the $k^M_*$-towers extend upwards.  Differentials respect weight, and $d_r$ has degree $(-1,r)$ in this representation.  If the cohomological dimension of $\kk$ is at most $1$, then this page is concentrated in slice degrees $n$ and $n+1$ and there are no nonzero targets for $d_r$, $r\ge 2$.}\label{fig:E2}
\end{figure}

\begin{theorem}\label{thm:odd}
If $\kk$ has odd characteristic, then the $\alpha_1$-periodic slice spectral sequence for $\sphere$ collapses with $\alpha_1^{-1}\slice_2=\alpha_1^{-1}\slice_\infty$.
\end{theorem}
\begin{proof}
Suppose $\kk$ has characteristic $p$ and write $i$ for the map $\Spec\kk\to\Spec\FF_p$.  The essentially smooth base change functor $i^*$ induces a map of spectral sequences $i^*:\alpha_1^{-1}\slice_2(\FF_p)\to \alpha_1^{-1}\slice_2(\kk)$ which is given by the extension of scalars map on $k^M_*$ and the identity on $\alpha_i$ for $i=1,4,5$.  Given the form of $\alpha_1^{-1}\slice_2(\kk)$, it suffices to show that $d_r^\kk \alpha_5=0$ for all $r\ge 2$, but $d_r^\kk\alpha_5 = i^*d_r^{\FF_p}\alpha_5$, and $d_r^{\FF_p}\alpha_5=0$ by \aref{thm:collapse}.  
\end{proof}

At this point, we know that if $\kk$ has odd characteristic or if $\cd_2 \kk\le 1$, then the $\alpha_1$-periodic slice spectral sequence collapses with $\alpha_1^{-1}\slice_2=\alpha_1^{-1}\slice_\infty\cong k^M_*[\alpha_1^{\pm 1},\alpha_4,\alpha_5]/\alpha_4^2$.  Paired with the conditional convergence portion of \aref{thm:etaSlices}, this implies that the spectral sequence in fact converges strongly to $\pi_\star \eta^{-1}\sphere$. In order to completely determine $\pi_\star\eta^{-1}\sphere$ for such $\kk$, we must resolve extension problems and understand the multiplicative structure. 

Suppose that $s\equiv 0$ or $3 \pmod{4}$, and consider the short exact sequences
\begin{equation}\label{eq:slice-filt-pi_s} 
  0 \to f_{q+1}\pi_{s}\eta^{-1}\sphere \to f_{q}\pi_{s}\eta^{-1}\sphere 
  \to k^M_q \to 0 
\end{equation}
obtained from the slice filtration and the determination of $\alpha_1^{-1}\slice_\infty$
for a field of odd characteristic.
Choose a lift $g_s\in f_0\pi_s\eta^{-1}\sphere=\pi_s\eta^{-1}\sphere$ 
of the nontrivial element in $k^M_0$, compatible with field extensions
from the prime field. If $s=0$,
$\pi_0\eta^{-1}\sphere$ is known to be the Witt ring by Morel's theorem,
and $g_0$ should be chosen as the unit. The slice filtration
on $\pi_0\eta^{-1}\sphere$ coincides with the filtration by powers
of the fundamental ideal $I$, as one deduces for example 
from \cite{Levine:filtration}. The
multiplicative structure on the slice filtration then 
supplies a natural
transformation to the sequence~(\ref{eq:slice-filt-pi_s})
from the short exact sequence
\[ 0 \to I^{q+1}\to I^q \to h^{q,q} \to 0 \]
solving Milnor's conjecture on quadratic forms \cite{OVV}.
The convergence statement~\aref{cor:fdconv} shows that this
natural transformation is an isomorphism for fields of
finite cohomological dimension. Since the constructions
involved commute with filtered colimits of fields, it is
thus an isomorphism for any field of odd characteristic.
In particular, the slice filtration is Hausdorff by
the main result of \cite{arason-pfister.krull}.

\begin{prop}\label{prop:main}
If $\kk$ has odd characteristic or if $\cd_2\kk\le 1$, then, as a ring,
\[
  \pi_\star \sc(\eta^{-1}\sphere) \cong W(\kk)[\eta^{\pm 1},\sigma,\mu]/(\sigma^2)
\]
where $|\eta|=\alpha$, $|\sigma|=3+4\alpha$, and $|\mu|=4+5\alpha$.  If additionally $\cd \kk<\infty$, then $\sc(\eta^{-1}\sphere)\simeq \eta^{-1}\sphere$ and this is a computation of the $\eta$-periodic homotopy groups of the motivic sphere spectrum.
\end{prop}
\begin{proof}
The additive structure (which is simply a copy of $W(\kk)$ in nonnegative simiplicial degrees congruent to $0$ or $3$ mod $4$) follows from the above filtration considerations.  There is no room for hidden extensions, so the result follows.
\end{proof}

\section{Characteristic $0$ fields}\label{sec:char0}

We now consider the $\alpha_1$-periodic slice spectral sequence over a general field $\kk$ of characteristic $0$.  We prove that for any $\kk$ this spectral sequence converges strongly to $\pi_\star \sc(\eta^{-1}\sphere)$.  Moreover, the spectral sequence over $\QQ$ completely determines the spectral sequence over $\kk$ in a manner that we make precise in \aref{thm:profile}.  This allows us to extend the conclusion of \aref{prop:main} to fields with $\cd_2\kk\le 2$ and to extensions of $\QQ(\sqrt{-1})$, resulting in \aref{thm:main}.  We conclude with a conjectural description of the differentials which we hope will inspire further work on this problem.

The structure of our argument is somewhat surprising.  After proving that $d_2^\QQ=0$, we are able to put strong restrictions on the form of the differentials which may appear in $\alpha_1^{-1}\slice(\QQ)$.  We then employ a theorem of Orlov-Vishik-Voevodsky \cite{OVV} to show that for arbitrary $\kk/\QQ$, the differentials in $\alpha_1^{-1}\slice(\kk)$ are of the same form.  The proscribed form of the differentials guarantees that Boardman's $RE_\infty=0$, whence strong convergence follows.  The primary obstruction to computing the differentials seems to be the lack of a good description of $\sc(\eta^{-1}\sphere)$.

We make some preliminary definitions in order to start our arguments.  Recall that $\alpha_1^{-1}\slice_1\cong \pi_\star H\FF_2[\alpha_1^{\pm 1},\alpha_3,\alpha_4]/\alpha_4^2$.  For $k\ge 0$, set $\alpha_{2k+1}:=\alpha_3^k\alpha_1^{1-k}$, and for $k\ge 2$ set $\alpha_{2k}:=\alpha_4\alpha_3^{k-2}\alpha_1^{2-k}$.  These classes are chosen so that $\bar{\alpha}_\ell\mapsto \alpha_\ell$ under the localization map $\slice_1\to \alpha_1^{-1}\slice_1$ for all $\ell\ne 2$.  Note that as a $\pi_\star H\FF_2[\alpha_1^{\pm 1}]$-module, $\alpha_1^{-1}\slice_1$ is generated by $1,\alpha_3,\alpha_4,\alpha_5,\alpha_6,\ldots$.  Also note (for the purposes of applying the Leibniz rule) that, up to multiplication by a unit, $\alpha_{4k+1}$ is the square of $\alpha_{2k+1}$.

\begin{lemma}\label{lemma:d2}
The $d_2$ differential in $\alpha_1^{-1}\slice(\QQ)$ is trivial.
\end{lemma}
\begin{proof}
It suffices to prove that $d_2^\QQ\alpha_5=0$.  We know that $d_2^\QQ\alpha_5\in k^M_2(\QQ)\{\alpha_1^3\alpha_4\}$.  Base change to $\QQ_p$ provides a comparison map $\alpha_1^{-1}\slice(\QQ)\to \alpha_1^{-1}\slice(\QQ_p)$.  Since $\cd_2(\QQ_p)=2$, \aref{cor:fdconv} implies that $\alpha_1^{-1}\slice(\QQ_p)$ converges strongly to $\pi_\star\eta^{-1}\sphere$.  Furthermore, every class in $k^M_2(\QQ)$ is detected in some $k^M_2(\QQ_p)$.\footnote{Indeed, \cite[Lemma A.1]{KM} tells us that the map $k^M_2(\QQ)\to k^M_2(\RR)\oplus\bigoplus_p k^M_2(\QQ_p)$ is injective and computed on components by quadratic Hilbert symbols.  Hilbert reciprocity then implies that $k^M_2(\QQ)\to \bigoplus_p k^M_2(\QQ_p)$ is injective as well.}  As such, the computations of Wilson \cite{Wilson} over $\QQ_p$ imply that $d_2^\QQ\alpha_5=0$.
\end{proof}

\begin{theorem}\label{thm:Qprofile}
There is a nondecreasing\footnote{In fact, the sequence is strictly increasing unless it is eventually constant at $\infty$.} sequence of extended integers $r_k\in \ZZ_{\ge 3}\cup \{\infty\}$ for $k\ge 2$ such that if $r_k<\infty$ then $d_{r_k}^\QQ\alpha_{2^k+1} = \rho^{r_k}\alpha_{2^k}\alpha_1^{r_k+1}$, and if $r_k=\infty$ then $\alpha_{2^k+1}$ is a permanent cycle in $\alpha_1^{-1}\slice(\QQ)$.  The rest of the differentials in $\alpha_1^{-1}\slice(\QQ)$ are determined by the Leibniz rule.
\end{theorem}

\begin{remark}\label{rmk:interpret}
The above theorem may be thought of in the following terms.  In the weight $n$ $\alpha_1$-periodic slice spectral sequence, the $4k$-column of $\alpha_1^{-1}\slice_2$ is, up to multiplication by some power of the unit $\alpha_1$, generated by $\alpha_{4k+1}$, and the $4k-1$-column is generated by $\alpha_{4k}$ in the same sense.  These columns are connected by $d_{r_{\nu_2(4k)}} = \cdot \rho^{r_{\nu_2(4k)}}$ on $k^M_*$, where $\nu_2$ is $2$-adic valuation.
\end{remark}

\begin{proof}
By \aref{lemma:d2} and \aref{thm:d1}, $\alpha_1^{-1}\slice_3\cong k^M_*[\alpha_1^{\pm 1},\alpha_4,\alpha_5]/\alpha_4^2$.  If the spectral sequence does not collapse, then the first nonzero differential is necessarily of the form $d_r\alpha_5 = x\alpha_4\alpha_1^{r+1}$ for some $r\ge 3$ and $x\in k^M_r(\QQ)$.  Since $k^M_r(\QQ) = \ZZ/2\{\rho^r\}$ for $r\ge 3$, we in fact have $d_r\alpha_5=\rho^r\alpha_4\alpha_1^{r+1}$.  Set $r_2$ equal to this $r$.  The $\alpha_1^{-1}\slice_{r_2+1}$-page then has $k^M_*[\alpha_1^{\pm 1}]/\rho^{r_2}$ in positive stems congruent to $3$ mod $8$ and ${}_{\rho^{r_2}}k^M_*[\alpha_1^{\pm 1}]$ in positive stems congruent to $4$ mod $8$ (where ${}_xk^M_* = \{y\in k^M_*\mid xy=0\}$ is the $x$-torsion in $k^M_*$); the $(r_2+1)$-page also continues to have $k^M_*[\alpha_1^{\pm 1}]$ in nonnegative stems congruent to $0$ or $7$ mod $8$, and is $0$ otherwise.

The potential targets of the ${}_{\rho^{r_2}}k^M_*[\alpha_1^{\pm 1}]$ terms are all $0$, hence these classes are permanent.  Thus the next nonzero differential in the spectral sequence (if one exists) is necessarily of the form $d_{r_3}\alpha_9 = \rho^{r_3}\alpha_8\alpha_1^{r_3+1}$.  The $\rho^{r_3}$-torsion terms in the $(r_3+1)$-page are again permanent, and the next differential is of the form $d_{r_4}\alpha_{2^{4}+1}=\rho^{r_4}\alpha_{2^4}\alpha_1^{r_4+1}$.  Proceeding inductively proves the theorem.
\end{proof}

We now abstract the behavior observed in \aref{thm:Qprofile} and show that it is in fact generic.

\begin{defn}
For a given field $\kk$, suppose that there is a nondecreasing sequence of extended integers $r_k\in \ZZ_{\ge 3}\cup \{\infty\}$ for $k\ge 2$ such that the differentials $d_{r_k}\alpha_{2^k+1}=\rho^{r_k}\alpha_{2^k}\alpha_1^{r_k+1}$ and the Leibniz rule determine $\alpha_1^{-1}\slice(\kk)$.  In this case, we call $\{r_2,r_3,\ldots\}$ the \emph{profile} of $\alpha_1^{-1}\slice(\kk)$ and say that $\alpha_1^{-1}\slice(\kk)$ is \emph{determined by the profile $\{r_k\}$}.
\end{defn}  

\begin{theorem}\label{thm:profile}
Let $\{r_k\}$ denote the profile of $\alpha_1^{-1}\slice(\QQ)$ (guaranteed to exist by \aref{thm:Qprofile}).  Then for any characteristic $0$ field $\kk$, $\alpha_1^{-1}\slice(\kk)$ is also determined by the profile $\{r_k\}$.
\end{theorem}

\begin{proof}
Consider the map of map of spectral sequences $i^*:\alpha_1^{-1}\slice(\QQ)\to \alpha_1^{-1}\slice(\kk)$ induced by essentially smooth base change along $\Spec\kk\to \Spec\QQ$.  We have
\[
  d_{r_k}^\kk \alpha_{2^k+1} = i^*d_{r_k}^\QQ \alpha_{2^k+1} = i^*\rho^{r_k}\alpha_{2^k}\alpha_1^{r_k+1} = \rho^{r_k}\alpha_{2^k}\alpha_1^{r_k+1}.
\]

It remains to show that ${}_{\rho^{r_k}}k^M_*[\alpha_1^{\pm 1}]\{\alpha_{2^k+1}\}$ supports no higher differentials. Invoking \cite[Theorem 3.3]{OVV}, we see that ${}_{\rho^{r_k}}k^M_*$ is generated in degree $1$ as a $k^M_*$-module, so it suffices to show that $d_r^\kk [u]\alpha_{2^k+1}=0$ for all $r>r_k$ and $u\in \kk^\times$ such that $[u]\rho^{r_k}=0\in k^M_{r_k+1}$.  Fix such a $u$ and consider the subextension $\kk/\QQ(u)/\QQ$.  Let $j:\Spec\QQ(u)\to \Spec\QQ$ denote the corresponding map with associated map of spectral sequences $j^*:\alpha_1^{-1}\slice(\QQ(u))\to \alpha_1^{-1}\slice(\kk)$.  Our argument now splits into two cases: $u$ algebraic, and $u$ transcendental.

First suppose that $u$ is algebraic, in which case $\QQ(u)$ is a number field.  Tate's theorem \cite[Theorem A.2]{KM} implies that $k^M_n(\QQ(u)) = \ZZ/2\{\rho^n\}$ or $0$ for $n\ge 3$ and we have already seen that $d_{r_k}^{\QQ(u)}\alpha_{2^k+1} = \rho^{r_k}\alpha_{2^k}\alpha_1^{r_k+1}$.  Recall that $r_k\ge 3$, so this differential kills $k^M_*$ classes at and above degree $r_k$.  In particular, for $r>r_k$ the target group for $d_r^{\QQ(u)}[u]\alpha_{2^k+1}$ is $0$ and hence the differential is $0$.  Finally, we see that
\[
  d_r^{\kk}[u]\alpha_{2^k+1} = j^*d_r^{\QQ(u)}[u]\alpha_{2^k+1} = 0
\]
as well, as desired.

Now suppose that $u$ is transcendental, in which case \cite[Theorem 2.3]{KM} implies that there is a split short exact sequence
\[
  0\to k^M_*\QQ\to k^M_*\QQ(u)\xrightarrow{\bigoplus \partial_\pi} \bigoplus_\pi k^M_{*-1}\QQ[u]/(\pi)\to 0
\]
where $\pi$ ranges over monic irreducible polynomials in $\QQ[u]$ and $\partial_\pi:k^M_*\QQ\to k^M_{*-1}\QQ[u]/(\pi)$ is the residue map taking $[\pi,u_2,u_3,\ldots,u_n]$ to $[u_2,\ldots,u_n]$.  In particular, for $n\ge 4$, $k^M_n\QQ(u)$ has $\FF_2$-basis consisting of $\rho^n$ and $[\pi]\rho^{n-1}$ for $\pi\in\QQ[u]$ monic irreducible.  Thus the differential $d_{r_k}^{\QQ(u)}\alpha_{2^k+1}=\rho^{r_k}\alpha_{2^k}\alpha_1^{r_k+1}$ kills $k^M_*(\QQ(u))[\alpha_1^{\pm 1}]\{\alpha_{2^k}\}$ in Milnor-degree $r_k+1$ and above.  It follows that $d_r^{\QQ(u)}[u]\alpha_{2^k+1}=0$ for $r>r_k$ and the same base change trick as in the previous paragraph implies that $d_r^\kk [u]\alpha_{2^k+1}=0$.  We conclude that $\alpha_1^{-1}\slice(\kk)$ is determined by the profile $\{r_k\}$.
\end{proof}

\begin{theorem}\label{thm:conv}
Let $\kk$ be any field of characteristic different from $2$.  Then $\alpha_1^{-1}\slice(\kk)$ converges strongly to $\pi_\star\sc(\eta^{-1}\sphere)$.  If $\cd \kk<\infty$, this target is isomorphic to $\pi_\star \eta^{-1}\sphere$.
\end{theorem}
\begin{proof}
We have already verified this result for odd characteristic fields and fields with finite cohomological dimension.  It remains to check characteristic $0$ fields of arbitrary cohomological dimension.  By \aref{thm:etaSlices}, we have weak convergence to $\sc(\eta^{-1}\sphere)$, so it suffices to check vanishing of Boardman's $RE_\infty$ term (for $E=\alpha_1^{-1}\slice(\kk)$).  By \cite[Remark after Theorem 7.1]{boardman}, it in turn suffices to check that for each tri-degree $(s,m+n\alpha)$ there are at most finitely many nonzero differentials $d_r:\alpha_1^{-1}\slice_r^{s,m+n\alpha}\to \alpha_1^{-1}\slice_r^{s+r,m-1+n\alpha}$.  By \aref{thm:profile}, $\alpha_1^{-1}\slice(\kk)$ has profile $\{r_k\}$ where $\{r_k\}$ is the profile of $\alpha_1^{-1}\slice(\QQ)$.  In particular, the finiteness condition on nonzero differentials is met, and we may conclude that we indeed have strong convergence.
\end{proof}

\begin{theorem}\label{thm:groups}
Suppose $\kk$ is a field of characteristic $0$ which has profile $\{r_k\}$.  Let $\nu_2$ denote $2$-adic valuation.  If $r_k<\infty$ for all $\kk$, then
\[
  \pi_m \sc(\eta^{-1}\sphere) \cong
  \begin{cases}
    W(\kk)&\text{if }m=0\text{;}\\
    W(\kk)/2^{r_k}&\text{if }m>0\text{, }m=4\ell-1\text{, and }k=\nu_2(4\ell)\text{;}\\
    {}_{2^{r_k}}W(\kk)&\text{if }m>0\text{, }m=4\ell\text{, and }k=\nu_2(4\ell)\text{;}\\
    0 &\text{otherwise}.
  \end{cases}
\]
If $\{r_k\}$ eventually takes the value $\infty$ with first instance $r_K=\infty$, then
\[
  \pi_m \sc(\eta^{-1}\sphere) \cong
  \begin{cases}
    W(\kk)&\text{if }m=0\text{;}\\
    W(\kk)/2^{r_k}&\text{if }m>0\text{, }m=4\ell-1\text{, and }k=\nu_2(4\ell)<K\text{;}\\
    {}_{2^{r_k}}W(\kk)&\text{if }m>0\text{, }m=4\ell\text{, and }k=\nu_2(4\ell)<K\text{;}\\
    W(\kk)&\text{if }m>0\text{, }m=4\ell-1\text{ or }4\ell\text{, and }\nu_2(4\ell)\ge K\text{;}\\
    0 &\text{otherwise}.
  \end{cases}
\]
\end{theorem}
\begin{proof}
This all follows from the slice filtration being the $I$-adic filtration, $\rho$ representing $2$ in $W(\kk)$, the structure of the differentials in \aref{thm:profile}, and strong convergence in \aref{thm:conv}.
\end{proof}

\begin{theorem}\label{thm:main}
Suppose that $\kk$ is not of characteristic $2$ and that $-1$ is a sum of four squares in $\kk$. Then, as a ring,
\[
  \pi_\star \sc(\eta^{-1}\sphere) \cong W(\kk)[\eta^{\pm 1},\sigma,\mu]/(\sigma^2)
\]
where $|\eta|=\alpha$, $|\sigma|=3+4\alpha$, and $|\mu| = 4+5\alpha$.  If additionally $\cd \kk<\infty$, then $\sc(\eta^{-1}\sphere)\simeq \eta^{-1}\sphere$ and this is a computation of the $\eta$-periodic homotopy groups of the motivic sphere spectrum.
\end{theorem}
\begin{proof}
By \aref{prop:main}, we may assume without loss of generality that $\chr \kk=0$ and that $-1$ is a sum of four squares in $\kk$.  It is standard that the latter condition is equivalent to $\rho^3=0\in k^M_3(\kk)$ (see \cite[Corollary X.6.20]{lam:intro}).  By \aref{thm:profile}, we see that the spectral sequence collapses (regardless of the profile of $\QQ$).  By \aref{thm:conv}, this proves the theorem.
\end{proof}

\begin{corollary}\label{cor:nonperiodic}
If $\cd \kk<\infty$ and $n\ge \max\{3m+5,4m\}$, then $\pi_{m+n\alpha}\sphere\cong \pi_{m+n\alpha}\eta^{-1}\sphere$; if, additionally, $\kk$ is not of characteristic $2$ and $-1$ is a sum of four squares in $\kk$, then these groups are $0$ or $W(\kk)$ according to whether $m\equiv 1,2\pmod{4}$ or $m\equiv 0,3\pmod{4}$, respectively.
\end{corollary}
\begin{proof}
All but the final statement was already observed in \cite[Theorem 5.5]{ORO:vanishing}.
\end{proof}

We certainly do \emph{not} expect that the $\alpha_1$-periodic slice spectral sequence collapses at its $E_2$ page in general.  Indeed, inspired by the computations of $\pi_\star\eta^{-1}\sphere\comp{2}$ by Guillou-Isaksen \cite{GI:etaR} over $\RR$ and Wilson \cite{Wilson} over $\QQ$, we make the following conjecture.

\begin{conjecture}\label{conj}
The $\alpha_1$-periodic slice spectral sequence over $\QQ$ has profile $\{3,4,5,\ldots\}$\emph{i.e.}, $r_k=k+1$ for all $k$.
\end{conjecture}

If \aref{conj} holds, then over $\kk$ of characteristic $0$,
\[
  \pi_m \sc(\eta^{-1}\sphere)\cong
  \begin{cases}
    W(\kk)&\text{if }m=0\text{,}\\
    W(\kk)/2^{\nu_2(4\ell)+1}&\text{if }m=4\ell-1>0\text{,}\\
    {}_{2^{\nu_2(4\ell)+1}}W(\kk)&\text{if }m=4\ell>0\text{,}\\
    0&\text{otherwise.}
  \end{cases}
\]
Curiously, this makes it appear as if $\sc(\eta^{-1}\sphere)$ might fit into a ``connective image of $J$ fiber sequence'' of the form $\sc(\eta^{-1}\sphere)\to \kw\to\Sigma^4\kw$ where $\kw$ is the connective cover of the $2$-complete Witt $K$-theory spectrum.  Over $\kk=\CC$, one may show that the cone on the map $\sigma_\infty:\Sigma^3\kw\to \eta^{-1}\sphere$ of \aref{thm:sigma_infty} coincides with $\kw$. In fact, the composition $\Sigma^3\kw \to \eta^{-1}\sphere\to \kw$ is zero, as one may deduce inductively, starting with the triviality of 
\[\Sigma^3\eta^{-1}\sphere \to \Sigma^3\kw\to\eta^{-1}\sphere \to \kw\]
and continuing along the cell presentation of $\kw$ given in the proof
of \aref{thm:sigma_infty}. Hence there is an induced map from the cone 
of $\Sigma^3\kw \to \eta^{-1}\sphere$ to $\kw$. This map induces an isomorphism on homotopy groups, hence is an equivalence by cellularity. In particular, one may express $\eta^{-1}\sphere$ over the complex numbers as the fiber of a map $\kw\to \Sigma^4\kw$. 

The Adams operations on the $2$-complete algebraic $K$-theory spectrum $\KGL\comp{2}$ consitute an action of $\ZZ_2^\times$, the units in the $2$-adic integers.  When $\kk$ has finite virtual cohomological dimension, the results of \cite{HKO:holim} imply that $\KQ\comp{2}\simeq (\KGL\comp{2})^{h\{\pm 1\}}$ inherits an action of $\ZZ_2^\times/\{\pm 1\}\cong \ZZ_2$ by Adams operations.  Inverting $\eta$ and taking the connective cover results in Adams operations on $\kw\comp{2}$.  For any such $\psi^g$, the difference of ring maps $\psi^g-1:\kw\comp{2}\to\kw\comp{2}$ lifts to a map $\psi^g-1:\kw\comp{2}\to\Sigma^4\kw\comp{2}$.  (This can be seen by observing that $\Sigma^4\kw\comp{2}$ is the $4$-connective cover of $\kw\comp{2}$ and the cofiber of $\Sigma^4\kw\comp{2}\to\kw\comp{2}$ is the Eilenberg-MacLane spectrum associated with the homotopy module $\ul W\comp{2}[\eta^{\pm 1}]$.)

The $\eta$-periodic unit $\eta^{-1}\sphere\comp{2}\to \kw\comp{2}$ factors through the fiber $\jw_g$ of $\psi^g-1:\kw\comp{2}\to\Sigma^4\kw$ because $\psi^g-1$ is a difference of ring maps.  This leads to the following conjecture, which is similar in spirit to Mahowald's presentation of the $v_1$-periodic sphere in topology.  

\begin{conjecture}\label{conj:jw}
The map $\eta^{-1}\sphere\comp{2}\to \jw_3$ induced by the $\eta$-periodic unit $\eta^{-1}\sphere\comp{2}\to\kw\comp{2}$ is an equivalence.
\end{conjecture}

Work in progress of Tom Bachmann and Mike Hopkins suggests that the action of $\psi^3$ on $HW\wedge \kw\comp{2}$ is such that the unit map smashed with $HW\comp{2}$ induces an equivalence $HW\comp{2}\to HW\wedge \jw_3$.  Since $\eta^{-1}\sphere$ is $HW$-complete, this would immediately prove that $\eta^{-1}\sphere\comp{2}\simeq \jw_3$.  It is presumably also the case that $(\psi^3-1)(\beta^k) = (9^k-1)\beta^k$ on $\pi_\star \kw\comp{2}$, in which case a comparison of slice spectral sequences would prove \aref{conj}.

\begin{remark}
The equivalence $\eta^{-1}\sphere\comp{2}\simeq \jw_3$ would also lead to a complete determination of the homotopy type and groups of $\eta^{-1}\sphere$.  Let $X_\kk$ denote the Harrison space of orderings of $\kk$.  Then $\pi_\star \{\eta^{-1},1/2\}\sphere \cong H^0(X_\kk;\pi_m^{\top}\sphere[1/2])$, which can be seen by the results of \cite{bachmann:rho}, a descent spectral sequence, and the fact that $X_\kk$ is a Stone space.  The 2-primary arithmetic fracture square would then imply that
\[
  \pi_m \eta^{-1}\sphere\cong
  \begin{cases}
    W(\kk)&\text{if }m=0\text{,}\\
    H^0(X_k;\pi_m^\top\sphere[1/2])&\text{if }m>0\text{ and }m\equiv 1\text{ or }2\pmod{4}\text{,}\\
    W(\kk)/2^{\nu_2(4\ell)+1}\oplus H^0(X_k;\pi_m^\top\sphere[1/2])&\text{if }m=4\ell-1>0\text{,}\\
    {}_{2^{\nu_2(4\ell)+1}}W(\kk)\oplus H^0(X_k;\pi_m^\top\sphere[1/2])&\text{if }m=4\ell>0\text{,}\\
    0&\text{otherwise.}
  \end{cases}
\]
\end{remark}

\bibliographystyle{plain}
\bibliography{etaPeriodic}

\end{document}